\documentclass[11pt]{article} 
\usepackage{setspace, ifdraft, blindtext}
\ifdraft{%
    \usepackage{todonotes}}{ 
    \usepackage[disable]{todonotes} 
}
\usepackage{kantlipsum}
\usepackage{tikz}
\usepackage{pgf,tikz}
\usetikzlibrary{arrows}
\usetikzlibrary[patterns]
\usepackage[T1]{fontenc}
\usepackage{enumitem}
\usepackage[english]{babel}
\usepackage{graphicx}

\usepackage{amsmath, amssymb, amsthm}

\theoremstyle{plain}
\newtheorem{theorem}{Theorem}[section]

\newtheorem*{prop*}{Proposition}
\newtheorem{lemma}[theorem]{Lemma}
\newtheorem*{lem*}{Lemma}
\newtheorem{corollary}[theorem]{Corollary}
\newtheorem{claim}[theorem]{Claim}

\theoremstyle{remark}
\newtheorem*{remark}{Remark}
\theoremstyle{definition}
\newtheorem{definition}[theorem]{Definition}
\makeatletter
\def\th@definition{%
  \thm@notefont{}
  \normalfont 
}
\makeatother

\usepackage[margin=1in]{geometry}
\newcommand{\beq}{\begin{equation}}
\newcommand{\eeq}{\end{equation}}

\usepackage{algorithm}
\usepackage[noend]{algpseudocode}
\usepackage{url}

\newcommand\cJ{{\mathcal J}}

\newcommand{\sm}{\setminus}

\title{Finding Cliques in Social Networks: A New Distribution-Free Model}

\author{
Jacob Fox\thanks{Department of Mathematics, Stanford University, Stanford, CA 94305. Email: {\tt jacobfox@stanford.edu}. Supported by a Packard Fellowship, by NSF Career Award DMS-1352121 and by an Alfred P. Sloan Fellowship.}
\and 
Tim Roughgarden\thanks{Department of Computer Science, Stanford University, Stanford, CA 94305. Email:  {\tt tim@cs.stanford.edu}.  Supported in part by NSF award CCF-1524062.}
\and
C. Seshadhri\thanks{Department of Computer Science, University of California, Santa Cruz, CA 95064. Email: {\tt sesh@ucsc.edu}}
\and
Fan Wei\thanks{Department of Mathematics, Stanford University, Stanford, CA 94305. Email: {\tt fanwei@stanford.edu}.}
\and 
Nicole Wein\thanks{Department of Computer Science, Massachusetts Institute of Technology, Cambridge, MA 02139. Email: {\tt nwein@mit.edu}. Supported in part by a National Science Foundation fellowship. Most of this work was done while the author was at Stanford University.}
}

\begin{document}
\date{\today}
\maketitle
\thispagestyle{empty}

\begin{abstract}
We propose a new distribution-free model of social networks.
Our definitions are motivated by one of the most universal signatures of social networks, triadic closure---the property that pairs of vertices with common neighbors tend to be adjacent. Our most basic definition is that of a \emph{$c$-closed} graph, where for every pair of vertices $u,v$ with at least $c$ common neighbors, $u$ and $v$ are adjacent. We study the classic problem of enumerating all maximal cliques, an important task in social network analysis. We prove that this problem is fixed-parameter tractable with respect to $c$ on $c$-closed graphs. 
Our results carry over to {\em weakly $c$-closed graphs}, which only require a vertex deletion ordering that avoids pairs of non-adjacent vertices with $c$ common neighbors.  Numerical experiments show that well-studied social networks with thousands of vertices tend to be weakly $c$-closed for modest values of $c$.
 \end{abstract}
 \vfill
\pagebreak
\setcounter{page}{1}
\section{Introduction}

There has been an enormous amount of important work over the past 15
years on models for capturing the special structure of social networks.
This literature is almost entirely driven by the quest for
generative (i.e., probabilistic) models.
Well-known examples of such models include preferential
attachment~\cite{BaAl99}, the copying model~\cite{KuRa+00}, Kronecker
graphs~\cite{ChZhFa04,LeChKlFa10}, and the Chung-Lu random graph model~\cite{ChLu02,ChLu02-2}. There is little consensus about which generative model is the
``right'' one.  For example, already in 2006, the survey by Chakrabarti and
Faloutsos~\cite{ChFa06} compares~23 different probabilistic models
of social networks, and multiple new such models are proposed every year.

Generative models articulate a hypothesis about what ``real-world''
social networks look like, how they are created, and how they will evolve in the future. They are directly useful for generating
synthetic data and can also be used as a proxy
to study the effect of random processes on a network ~\cite{AlJeBa00,LiAm+08,MoSa10}. 
However, the plethora of models presents a quandary for the design of algorithms for social networks with rigorous guarantees: which of these models should one tailor an algorithm to?  One idea is to seek algorithms that are tailored to {\em none} of them, and to instead assume only determinstic combinatorial conditions that share the spirit of the prevailing generative models.  This is the approach taken in this paper.

There is empirical evidence that 
many NP-hard optimization problems 
are often easier to solve in social networks
than in worst-case graphs.  
For example, lightweight
heuristics are unreasonably effective in practice for finding the
maximum clique of a social network~\cite{RoGlGe15}. Similar success stories have been repeatedly reported for the problem
of recovering dense subgraphs or communities in social networks~\cite{Tsourakakis13,www15,MiPaPe+15,Tsourakakis15}. To define our 
notion of ``social-network-like'' graphs, we turn to one of the most agreed upon properties of social networks---{\em triadic closure},
the property that when two members of a social network have a friend in common, they are likely to be friends themselves.

\subsection{Properties of social networks}
There is wide consensus that social networks have relatively predictable
structure and features, and accordingly are not well modeled by
arbitrary graphs.
From a structural viewpoint, the most well studied and empirically
validated statistical properties of social networks
include heavy-tailed degree 
distributions~\cite{BaAl99,BrKu+00,FFF99},
a high density of triangles~\cite{WaSt98,SaCaWiZa10,UgKa+11}
and other dense subgraphs or
``communities''~\cite{For10,GiNe02,Ne03a,Ne06,LeLaDaMa08}, low diameter
and the small world property~\cite{Kl00,Kl00-2,Kl02,Ne01}, and triadic closure ~\cite{SaCaWiZa10,UgKa+11,SePiKo13}.

For the problem of finding cliques in networks, it does not help to assume that the graph has small diameter (every network can be rendered
small-diameter by adding one extra vertex connected to all other
vertices).  
Similarly, merely assuming a power-law degree
distribution does not seem to make the clique problem easier~\cite{FePaPa06}. 
On the other hand, as we show, the clique problem is tractable on graphs with strong triadic closure properties.

\subsection{Our model: $c$-closed graphs}\label{ss:closed}

Motivated by the empirical evidence for triadic closure in social networks, we define the class of $c$-closed graphs. Figure~\ref{f:enron} shows the triadic closure of the network of email communications at Enron~\cite{enron} and other social networks have been shown to behave similarly~\cite{BK13}. In particular, the more common neighbors two vertices have, the more likely they are to be adjacent to each other. The definition of $c$-closed graphs is a coarse version of this property: we assert that \emph{every} pair of vertices with $c$ or more common neighbors must be adjacent to each other.

\begin{figure}[h]
\centering
\includegraphics[width=.5\textwidth]{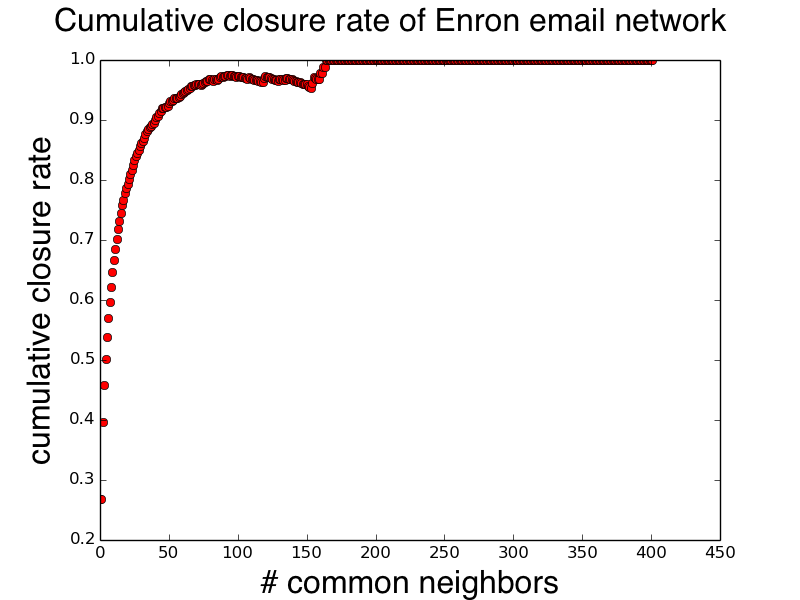}

\caption{Triadic closure properties of the Enron email graph (36K nodes and 183K edges). Nodes of
  this  network are Enron employees, and there is an edge
  connecting two employees if one sent at least
  one email to the other. Given an $x$ value, the
  $y$-axis shows the cumulative closure rate: the fraction of pairs of nodes with \emph{at least} $x$ common neighbors that are themselves
  connected by an edge.}  
  \label{f:enron}
\end{figure}

\begin{definition}[$c$-closed graph]\label{d:closed}
For a positive integer~$c$, an undirected graph $G=(V,E)$ is {\em $c$-closed} if, whenever two distinct vertices $u,v \in V$ have at least $c$ common
neighbors, $(u,v)$ is an edge of $G$.  
\end{definition}

The parameter~$c$ interpolates between a disjoint unions of cliques (when $c=1$) and all graphs (when $c=|V|-1$). The class of 2-closed graphs is already non-trivial.  These are exactly the graphs that do not contain $K_{2,2}$ or a diamond ($K_4$ minus an edge) as an induced subgraph.  For example, graphs with girth at least~5, e.g.~constant-degree expanders, are 2-closed. For every $c$, membership in the class of $c$-closed graphs can be
checked by squaring the adjacency matrix in $O(n^{\omega})$ time, where $\omega < 2.373$ is the matrix multiplication exponent.

While the definition of $c$-closed captures important aspects of triadic closure, it is fragile in the sense that a single pair of non-adjacent vertices with many common neighbors prevents the graph from being $c$-closed for a low value of $c$. To address this, we define the more robust notion of \emph{weakly $c$-closed} graphs and show that our results carry over to these graphs.  Well-studied social networks with thousands of vertices are typically weakly $c$-closed for 
modest values of $c$ (see Table~\ref{tab}). 

\begin{definition}
Given a graph and a value of $c$, a \emph{bad pair} is a non-adjacent pair of vertices with at least $c$ common neighbors.
\end{definition}

\begin{definition}[Weakly $c$-closed graph]
A graph is \emph{weakly $c$-closed} if there exists an ordering of the vertices $\{v_1,v_2,\dots,v_n\}$ such that for all $i$, $v_i$ is in no bad pairs in the graph induced by $\{v_i,v_{i+1},\dots,v_n\}$.
\end{definition}

A graph can be $c$-closed only for large $c$ but weakly $c$-closed for much smaller $c$. Consider the graph $G$ that is a clique of size $k$ with one edge $(u,v)$ missing. $G$ is not $c$-closed for any $c<k-2$. The only bad pair in $G$ is $(u,v)$. The vertex ordering that places $u$ and $v$ at the end demonstrates that $G$ is weakly 1-closed. Also, the properties of being $c$-closed and weakly $c$-closed are hereditary, meaning that they are closed under taking induced subgraphs. We will use this basic fact often. 

\subsection{Our contributions}

One can study a number of computational problems on $c$-closed (and weakly $c$-closed) graphs. We focus on the problem of enumerating all maximal cliques, an important problem in social network analysis~\cite{conte2016finding,tomita, du2006parallel,eppstein2011listing,tomita2007efficient}. 
We study \emph{fixed-parameter tractability}\footnote{A problem is said to be fixed-parameter tractable with respect to a parameter $k$ if there is an algorithm that solves it in time at most $f(k)n^\alpha$ where $f$ can be an arbitrary function but $\alpha$ is a constant.} with respect to $c$. There is a rich literature on fixed-parameter tractability for other graph parameters including treewidth, arboricity, and the size of the output~\cite{cygan2015parameterized}.

In a graph $G$, a \emph{clique} is a subgraph of $G$ in which there is an edge between every pair of vertices. A \emph{maximal} clique in $G$ is a clique that cannot be made any larger by the addition of some other vertex in $G$. In any graph, all maximal cliques can be listed in $O(mn)$ time per maximal clique~\cite{tsukiyama1977new}. We focus on the following two problems:
\begin{enumerate}
\item determining the maximum possible number of maximal cliques in a $c$-closed graph on $n$ vertices.
\item finding algorithms to enumerate all maximal cliques in $c$-closed graphs (that run faster than $O(mn)$ time per maximal clique).
\end{enumerate}

Our main result is that for constant $c$ the number of maximal cliques in a $c$-closed graph on $n$ vertices is $O(n^{2-2^{1-c}})$. More specifically, we prove the following bound.

\begin{theorem}\label{intro:bound}Any $c$-closed graph on $n$ vertices has at most $\min\{3^{(c-1)/3} n^2, 4^{(c+4)(c-1)/2}n^{2-2^{1-c}}\}$ maximal cliques.
\end{theorem}

For example, 3-closed, 4-closed, and 5-closed graphs have $O(n^{3/2})$, $O(n^{7/4})$, and $O(n^{15/8})$ maximal cliques respectively.

The proof of the first bound listed in Theorem~\ref{intro:bound} extends to weakly $c$-closed graphs, giving the following result.

\begin{theorem}\label{intro:weak}
Any weakly $c$-closed graph on $n$ vertices has at most $3^{(c-1)/3} n^2$ maximal cliques.
\end{theorem}

In Appendix~\ref{app:weak}, we give experimental results showing that well-studied social networks are weakly $c$-closed for modest values of $c$. Note that Theorem~\ref{intro:weak} is exponential in the even smaller value of $(c-1)/3$.

Since in any graph all maximal cliques can be listed in $O(mn)$ time per maximal clique, Theorem~\ref{intro:bound} proves that listing all maximal cliques in a $c$-closed graph is fixed-parameter tractable (i.e. has running time $f(c)n^\alpha$ for constant $\alpha$). We give an algorithm for listing all maximal cliques in a $c$-closed graph that runs faster than applying the $O(mn)$-per-clique algorithm as a black box. Our algorithm follows naturally from the proof of Theorem~\ref{intro:bound} and gives the following theorem, where $p(n,c)$ denotes the time to list all wedges (induced 2-paths) in a $c$-closed graph on $n$ vertices. A result of G\k{a}sieniec, Kowaluk, and Lingas~\cite{gkasieniec2009faster} implies that $p(n,c)= O(n^{2+o(1)}c + c^{(3-\omega-\alpha)/(1-\alpha)}n^\omega  + n^\omega \log n)$ where $\omega$ is the matrix multiplication exponent and $\alpha>0.29$.

\begin{theorem}\label{intro:alg}
In any $c$-closed graph,
a set of cliques containing all maximal cliques can be generated in time $O(p(n,c)+3^{c/3}n^2)$. The exact set of all maximal cliques in any $c$-closed graph can be generated in time $O(p(n,c)+3^{c/3}2^ccn^2)$.
\end{theorem}


Non-trivial lower bounds for the number of maximal cliques in a $c$-closed graph were previously known only for extreme values of $c$. A 2-closed graph can have $n^{3/2}$ maximal cliques~\cite{eschen2011graphs}. The classic Moon-Moser graph (with additional isolated vertices) is $(n-2)$-closed and has $3^{\lfloor n/3\rfloor}$ maximal cliques~\cite{moon1965cliques}. This graph consists of the complete multipartite graph with $\lfloor n/3 \rfloor$ parts of size $3$, and possibly additional isolated vertices.
By taking a disjoint union of $n/(c+2)$ Moon-Moser graphs on $(c+2)$ vertices, we can construct a $c$-closed graph on $n$ vertices with $\Omega(c^{-1}3^{c/3}n)$ maximal cliques for all $n \geq c$.
We give improved lower bounds for intermediate values of $c$.

\begin{theorem}\label{intro:lb} 
For any positive integer $c$, there are $c$-closed graphs with $n$ vertices and 
$\Omega(c^{-3/2}2^{c/2}n^{3/2})$ maximal cliques.
\end{theorem}

It is an open problem to determine the exact exponent of $n$ (between $3/2$ and $2-2^{1-c}$) in the expression for the maximum number of maximal cliques in a $c$-closed graph.


\subsection{Related work}
There are only a few algorithmic results for graph classes motivated by social networks. Although a number of NP-hard problems remain NP-hard on graphs with a power-law degree distribution~\cite{Ferrante2007}, several problems in P have been shown to be easier on such graphs. Brach, Cygan, Lacki, and Sankowski~\cite{brach2016algorithmic} give faster algorithms for transitive closure, maximum matching, determinant, PageRank and matrix inverse. Borassi, Crescenzi, and Trevisan~\cite{borassi2016axiomatic} assume several axioms satisfied by real-world graphs, one being a power-law degree distribution, and give faster algorithms for diameter, radius, distance oracles, and computing the most ``central'' vertices. Motivated by triadic closure, Gupta, Roughgarden, and Seshadhri~\cite{gupta2016decompositions} define \textit{triangle-dense} graphs and prove relevant structural results. Intuitively, they prove that if a constant fraction of two-hop paths are closed into triangles, then the graph must contain many dense clusters.

For general graphs, Moon and Moser prove that the maximum possible number of maximal cliques in a graph on $n$ vertices is $3^{n/3}$ (realized by a complete $n/3$-partite graph)~\cite{moon1965cliques}. Tomita, Tanaka, and Takahashi prove that the time to generate all maximal cliques in any $n$-vertex graph is also $O(3^{n/3})$~\cite{tomita}.

The clique problem has been studied on 2-closed graphs (under a different name). Eschen, Hoang, Spinrad, and Sritharan~\cite{eschen2011graphs} show that the maximum number of maximal cliques in a 2-closed graph is $O(n^{3/2})$. They also show a matching lower bound via a projective planes construction. Suppose $n=p^2+p+1$ for a positive integer $p$ and consider a finite projective plane on $n$ points (and hence with $n$ lines, see e.g.~\cite{projective_planes}). Let $G$ denote the bipartite graph representing the point-line incidence matrix.  The defining properties of finite projective planes imply that no two vertices have two common neighbors, so the 2-closed condition is vacuously satisfied.  Every vertex of~$G$ has degree $p+1$, so the graph has $\Theta(n^{3/2})$ edges, each a maximal clique. 

The clique problem has also been studied on other special classes of graphs such as graphs embeddable on a surface~\cite{dujmovic2011maximum} and graphs of bounded degeneracy~\cite{eppstein2010listing}. Degeneracy is a measure of everywhere sparsity. More formally, the degeneracy of a graph $G$ is the smallest value $d$ such that every nonempty subgraph
of $G$ contains a vertex of degree at most $d$. Eppstein et al. show that the maximum number of maximal cliques in a graph of degeneracy $d$ is $O(n3^{d/3})$. The degeneracy of a graph, however, can be much larger than its $c$-closure. For example, the degeneracy of a graph is at least the size of a maximum clique, while even in 1-closed graphs, the size of the maximum clique can be arbitrarily large.

Clique counting is a classical problem in extremal combinatorics. One fundamental question is to count the \emph{minimum} number of cliques in graphs with fixed number of edges i.e. to show that graphs with few cliques must have few edges. This simple question turns out to be a complex problem, and is settled for $K_3$ by Razborov~\cite{Raz} by flag algebra, $K_4$ by Nikiforov~\cite{Niki} by a combination of combinatorics and analytical arguments, and
all $K_t$ by Reiher~\cite{Reiher} by generalizing the argument of flag algebra to all sizes of cliques. 

There has also been a long line of work in combinatorics on counting (not necessarily maximal) cliques in graphs with certain excluded subgraphs, subdivisions, or minors. Most recently, Fox and Wei give an asymptotically tight bound on the maximum number of cliques in graphs with forbidden minors~\cite{fox2017number}, and an upper bound on the maximum number of cliques in graphs with forbidden subdivisions or immersions~\cite{fox2016number}. 

Many problems in combinatorics can be phrased as counting the number of cliques or independent sets in a (hyper)graph. For example, the problems of finding the volume of the metric polytope and counting the number of $n$-vertex $H$-free graphs (for any fixed subgraph $H$) can be translated into clique counting problems. The recently developed ``container method''~\cite{container1,container2} is a powerful tool to bound the number of cliques in (hyper)graphs and can be used to tackle a great range of problems. 

\subsection{Organization}

In Section~\ref{sec:init} we prove the first bound listed in Theorem~\ref{intro:bound}, state Theorem~\ref{intro:weak}, and introduce the proof of Theorem~\ref{intro:alg}. In Section~\ref{sec:imp} we prove the second bound listed in Theorem~\ref{intro:bound} (which has improved dependence on $n$). In Section~\ref{sec:lb} we prove Theorem~\ref{intro:lb}.

In Appendix~\ref{app:alg} we give the full proof of Theorem~\ref{intro:alg}. In Appendix~\ref{app:weak} we further discuss weakly $c$-closed graphs including relevant experimental results. In Appendix~\ref{app:general}, we give generalizations of $c$-closed and weakly $c$-closed graphs and extensions of and Theorem~\ref{intro:bound}. In Appendix~\ref{app:kk} we give a preliminary result regarding the number of maximal cliques in $c$-closed $K_k$-free graphs.

\subsection{Notation}

All graphs $G(V,E)$ are simple, undirected and unweighted. For any $v\in V$, let $N(v)$ denote the neighborhood of $v$. When the current graph is ambiguous, $N_G(v)$ will denote the neighborhood of $v$ in $G$. For any $S\subseteq V$, let $G[S]$ denote the subgraph of $G$ induced by $S$. 

\section{Initial Bound and Algorithm}\label{sec:init}

\subsection{Bound on number of maximal cliques}
In this section, we prove the following bound on the number of maximal cliques in a $c$-closed graph and show that this bound carries over to weakly $c$-closed graphs. Let $F(n,c)$ denote the maximum possible number of maximal cliques in a $c$-closed graph on $n$ vertices. The following theorem uses a natural peeling process and obtain an initial upper bound on the number of maximal cliques. A more involved analysis, Theorem \ref{thm:imp} which gives a tighter upper bound, is delayed to later.
\begin{theorem}[restatement of part of Theorem~\ref{intro:bound}]\label{thm:init}
For all positive integers $c, n$, we have $F(n,c) \leq 3^{(c-1)/3} n^2$.
\end{theorem}
\begin{proof}
Let $G$ be a $c$-closed graph on $n$ vertices and let $v \in V(G)$ be an arbitrary vertex. Every maximal clique $K\subseteq G$ is of one of the following types: 
\begin{enumerate}
\itemsep0em 
\item The clique $K$ does not contain vertex $v$; and $K$ is maximal in $G \setminus \{v\}$. 
\item The clique $K$ contains vertex $v$; and $K \setminus \{v\}$ is maximal in $G \setminus \{v\}$.
\item The clique $K$ contains vertex $v$; and $K \setminus \{v\}$ is not maximal in $G \setminus \{v\}$.
\end{enumerate}


Bounding the number of maximal cliques of type 1 and 2 is straightforward because every such clique can be obtained by starting with a clique maximal in $G \setminus \{v\}$ and extending it to include vertex $v$ if possible. Therefore, the number of maximal cliques of types 1 and 2 combined is at most $F(n-1,c)$.

Type 3 cliques are maximal in $N(v)$, but not in $G\sm\{v\}$. 
We will prove that the number of maximal cliques of type 3 is at most $3^{(c-1)/3}n$, crucially using the $c$-closed property. Figure~\ref{f:case3} shows a maximal clique $K$ of type 3. 
\begin{figure}[ht]
\centering
\includegraphics[width=.4\textwidth]{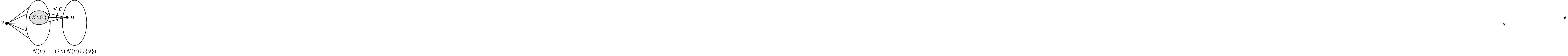}
\caption{A maximal clique $K$ of type 3: $K$ contains vertex $v$ and $K \setminus \{v\}$ is not maximal in $G \setminus \{v\}$. Property C asserts that there exists a vertex $u\not\in N(v)$ whose neighborhood contains $K\setminus \{v\}$. Since $G$ is $c$-closed, $|N(u)\cap N(v)|<c$.}
\label{f:case3}
\end{figure} 

We claim that each type 3 maximal clique $K$ satisfies the following three properties.

\begin{enumerate}[label=\Alph*)] 
\itemsep0em 
\item $K \setminus \{v\}$ is a clique in the neighborhood of $v$, and 
\item $K \setminus \{v\}$ is not in the neighborhood of any other vertex in $N(v)$.
\item There exists a vertex $u\not\in N(v)$ whose neighborhood contains $K\setminus \{v\}$.
\end{enumerate}
Property A is clear since $K$ is a clique containing $v$. Property B is true because if we can extend $K \setminus \{v\}$ to include some vertex $w\in N(v)$, then $K$ can also be extended to include $w$ which contradicts the fact that $K$ is maximal. To see property C, note that since $K \setminus \{v\}$ is not a maximal clique in $G \setminus \{v\}$ we can extend the clique $K \setminus \{v\}$ to include some vertex in $G \setminus \{v\}$. By property B, we can extend $K \setminus \{v\}$ to include some vertex $u$ not in $N(v)$. 

Let $u$ be as in property C. Then $K\sm\{v\}$ must be a maximal clique in $G[N(v) \cap N(u)]$ because otherwise we could extend $K \setminus \{v\}$ to some other vertex in $N(v) \cap N(u)$, which contradicts property B.

Thus, the number of type 3 maximal cliques is at most 
\begin{equation}  \sum_{u\in G\setminus (N(v)\cup \{v\})} F(|N(u)\cap N(v)|,c).  \label{eq:simplerec}
\end{equation}
Since $G$ is $c$-closed, $|N(u)\cap N(v)|<c$ for all vertices $u\not\in N(v)$. Then since any $k$-vertex graph has at most $3^{k/3}$ maximal cliques~\cite{moon1965cliques}, \[F(|N(u)\cap N(v)|,c)\leq 3^{(c-1)/3}.\] Thus, the number of type 3 maximal cliques in $G$ is at most $3^{(c-1)/3} n$.

Counting all three types of maximal cliques, we have the following recursive inequality: 
\[ F(n,c)\leq F(n-1,c)+3^{(c-1)/3}n. \]
By induction on $n$ with the base case $F(1,c)=1$, this gives \[F(n,c)\leq 3^{(c-1)/3}{n+1 \choose 2} \leq 3^{(c-1)/3}n^2.\]
\end{proof}

Note that $v$ was chosen arbitrarily and the proof is valid as long as ``$|N(u)\cap N(v)|<c$ for all vertices $u\not\in N(v)$''. Thus, in each recursive level, we
only require the existence of a vertex $v$ in no bad pairs. Equivalently, it suffices to have an ordering of the vertices $\{v_1,v_2,\dots,v_n\}$ such that for all $i$, $v_i$ is in no bad pairs in the graph induced by $\{v_i,v_{i+1},\dots,v_n\}$. This is exactly the definition of a weakly $c$-closed graph. Thus, we get the following theorem.
\begin{theorem}[Restatement of Theorem~\ref{intro:weak}]
For any positive integers $c, n$, there are at most $3^{(c-1)/3} n^2$ maximal cliques in an $n$-vertex weakly $c$-closed graph.
\end{theorem}

\subsection{Algorithm to generate all maximal cliques}
Recall that $p(n,c)$ denotes the time to list all wedges (induced 2-paths) in a $c$-closed graph on $n$ vertices. A result of G\k{a}sieniec, Kowaluk, and Lingas~\cite{gkasieniec2009faster} implies that $p(n,c)= O(n^{2+o(1)}c + c^{(3-\omega-\alpha)/(1-\alpha)}n^\omega  + n^\omega \log n)$ where $\omega$ is the matrix multiplication exponent and $\alpha>0.29$.

\begin{theorem}[restatement of Theorem~\ref{intro:alg}]\label{thm:alg}
A superset of the maximal cliques in any $c$-closed graph can be generated in time $O(p(n,c)+3^{c/3}n^2)$. The exact set of all maximal cliques in any $c$-closed graph can be generated in time $O(p(n,c)+3^{c/3}2^ccn^2)$.
\end{theorem}

The algorithm follows naturally from the proof of Theorem~\ref{thm:init} with two additional ingredients:
\begin{itemize}
\item A preprocessing step to enumerate all wedges in the graph speeds up the later process of finding the intersection of the neighborhoods of two vertices (i.e. $N(u) \cap N(v)$ from the proof of Theorem~\ref{thm:init}). 
\item An algorithm of Tomita, Tanaka, and Takahashi~\cite{tomita} generates all maximal cliques in any $n$-vertex graph in time $O(3^{n/3})$. We apply this to the recursive calls on the small induced subgraphs $G[N(u) \cap N(v)]$, which have less than $c$ vertices, that arise in handling the type 3 cliques in the proof of Theorem~\ref{thm:init}.
\end{itemize}

We defer the full algorithm description and runtime analysis to Appendix~\ref{app:alg}.

\section{Improved Bound}\label{sec:imp}

Recall that $F(n,c)$ is the maximum number of maximal cliques in a $c$-closed graph on $n$ vertices. 

\begin{theorem}[restatement of part of Theorem~\ref{intro:bound}]\label{thm:imp} For all positive integers $c, n$, we have 
 $F(n,c) \leq 4^{(c+4)(c-1)/2}n^{2-2^{1-c}}$.
\end{theorem}
The structure of the proof is similar to that of Theorem \ref{thm:init}.
We get an improved bound by a separate analysis depending on whether $G$ has a vertex of ``high'' degree. This idea appears in the result of Eschen et al.~\cite{eschen2011graphs}, who prove the result for the $c=2$ case. 

We will require the following simple lemma. 
\begin{lemma}\label{lem:nbhd}
For any $v$, $G[N(v)]$ is a $(c-1)$-closed graph.
\end{lemma}
\begin{proof} Consider pair $x,y \in N(v)$ with $c-1$ common neighbors in $G[N(v)]$. Since vertex $v \notin N(v)$ is also a common neighbor of $x$ and $y$, $x$ and $y$ have $c$ common neighbors in $G$. Thus, $(x,y)$ is an edge.
\end{proof}
\begin{proof}[Proof of Theorem~\ref{thm:imp}]
Let $G=(V,E)$ be a $c$-closed graph on $n$ vertices with $F(n,c)$ maximal cliques. 
Let $\Delta(G)$ be the maximum degree of $G$.
\paragraph*{Case 1: \normalfont $\Delta(G) \leq n^{1/2}$.}
By Lemma~\ref{lem:nbhd} for all $v\in V(G)$, $G[N(v)]$ is $(c-1)$-closed. Then, since the number of maximal cliques containing $v$ is exactly the number of maximal cliques in $G[N(v)]$, we have 
\[ F(n,c)\leq nF(\Delta(G),c-1).\]
\paragraph*{Case 2: \normalfont $\Delta(G) > n^{1/2}$.}
Let $v\in V$ be a vertex of degree $\Delta(G)$. We will count the maximal cliques containing at least one vertex in $N(v)\cup\{v\}$, delete $N(v)\cup\{v\}$, and recurse. 

Since the number of maximal cliques containing $v$ is exactly the number of maximal cliques in $G[N(v)]$, $v$ is in at most $F(\Delta(G), c-1)$ maximal cliques. 
It remains to bound the number the maximal cliques that contain some vertex in $N(v)$ but not $v$ itself. Such a clique must contain some vertex in $u\in V\sm(N(v)\cup\{v\})$ (otherwise, it would not be maximal).
Let $\mathcal{K}$ be the set of such cliques. We will bound $|\mathcal{K}|$ by grouping the maximal cliques $K$ in $\mathcal{K}$ based on which vertices of $N(v)$ are in $K$. For nonempty $S \subseteq N(v)$, let $N(S)$ denote $\bigcap_{u\in S}N(u)$. Also, let $N_2(v)$ denote the set of vertices of distance \emph{exactly} 2 from $v$. Let us bound the number of cliques $K \in \mathcal{K}$ such that $K \cap N(v) = S$. The other vertices in $K$ must be in $N(S)\cap N_2(v)$. By Lemma~\ref{lem:nbhd}, $G[N(S)\cap N_2(v)]$ is a $(c-1)$-closed graph. The number of cliques $K \in \mathcal{K}$ such that $K \cap N(v) = S$ is at most $F(|N(S)\cap N_2(v)|,c-1).$  Summing over all subsets $S \subseteq N(v)$, we have
\begin{align}
|\mathcal{K}|\leq \sum_{S\subseteq N(v)} F(|N(S)\cap N_2(v)|,c-1).\label{C}
\end{align} 

For all $u\in N_2(v)$, since $u$ and $v$ are not adjacent, $|N(u)\cap N(v)|<c$ (because $G$ is $c$-closed). Each vertex in $N_2(v)$ can be in $N(S)$ for only $2^{c-1}$ sets $S\subseteq N(v)$, implying
\begin{align}
\sum_{S\subseteq N(v)} |N(S)\cap N_2(v)|&\leq |N_2(v)|2^{c-1} \leq \min\{(\Delta(G))^2,n\}2^{c-1}.\label{S}
\end{align}
We want to determine for all $S$ the value of $|N(S)\cap N_2(v)|$ that maximizes the upper bound for $|\mathcal{K}|$ in Inequality~(\ref{C}) subject to the constraint in Inequality~(\ref{S}). Later, we prove our bound on $F$ (from the theorem statement) by induction on $n$ and $c$. In fact, we show that $F(n,c)$ is bounded by 
\[ F_0(n,c)=4^{(c+4)(c-1)/2}n^{2-2^{1-c}},\]
the desired upper bound for $F(n,c)$ that we are trying to prove by induction. Since $F_0(n,c)$  is convex in $n$, by the inductive hypothesis we can apply Jenson's inequality on Inequality~(\ref{C}). Jenson's inequality implies that the upper bound on $|\mathcal{K}|$ is maximized by setting $|N(S)\cap N_2(v)|$ to be as large as possible (note that it cannot exceed $\Delta(G)$) for as many $S$ as possible until the bound in Inequality~(\ref{S}) is met and setting the rest to be 0. 
By Inequality~(\ref{S}), the number of non-zero terms $|N(S)\cap N_2(v)|$ we sum over is at most $\Delta(G)^{-1}\min\{(\Delta(G))^2,n\}2^{c-1}\leq \min\{\Delta(G),\frac{n}{\Delta(G)}\} 2^{c-1}$. Thus, we have the following continuation of Inequality~(\ref{C}).
 \begin{align}
|\mathcal{K}|&\leq\sum_{S\subseteq N(v)} F(|N(S)\cap N_2(v)|,c-1)
\leq \min\{\Delta(G),\frac{n}{\Delta(G)}\} 2^{c-1} F_0(\Delta(G),c-1)\label{E}.
 \end{align}

Recall $|\mathcal{K}|$ is the number of maximal cliques that contain some vertex in $N(v)$ but not $v$ itself, so we combine Inequality~(\ref{E}) with the observation (from the beginning of case 2) that $v$ is in at most $F(\Delta(G), c-1)$ maximal cliques to conclude that the number of maximal cliques containing at least one vertex in $N(v)\cup\{v\}$ is at most \[F_0(\Delta(G), c-1)(1+\min\{\Delta(G),\frac{n}{\Delta(G)}\} 2^{c-1})\leq F_0(\Delta(G), c-1)\min\{\Delta(G),\frac{n}{\Delta(G)}\} 2^c.\] Then, recursing on $G\setminus(N(v)\cup\{v\})$, we have:
 \[F(n,c) < F_0(\Delta(G), c-1)\min\{\Delta(G),\frac{n}{\Delta(G)}\} 2^c+F(n-\Delta(G),c).\]

Combining the low and high degree bounds on $F(n,c)$, we get the following recurrence. 
\begin{align*}
&F(n,1)=n, \ \ \ F(1,c)=1\\
&F(n,c)\leq \begin{cases} 
      nF(\Delta(G),c-1) & \Delta(G)\leq n^{1/2} \\
      F_0(\Delta(G), c-1)\frac{n}{\Delta(G)} 2^c+F(n-\Delta(G),c) & \Delta(G)> n^{1/2} \\
   \end{cases}
\end{align*}
The remainder of the proof shows inductively that the recurrence implies the desired bound $F(n, c)  \leq F_0(n,c) \leq n^{2- 2^{1-c}} 2^{(c+4)(c-1)/2}$. The desired bound holds in the two base cases. For the inductive case, we need to show that
\[n^{2- 2^{1-c}} 2^{(c+4)(c-1)/2}\geq \begin{cases} 
      n\Delta(G)^{2- 2^{2-c}} 2^{(c+3)(c-2)/2} & \Delta(G)\leq n^{1/2} \\
      \Delta(G)^{2- 2^{2-c}} 2^{(c+3)(c-2)/2}\frac{n}{\Delta(G)} 2^c\\\hspace{5mm}+(n-\Delta(G))^{2- 2^{1-c}} 2^{(c+4)(c-1)/2} & \Delta(G)> n^{1/2}.\end{cases}\]
      
In the $\Delta(G)\leq n^{1/2}$ case, the expression is maximized when $\Delta(G)=n^{1/2}$. Thus, 
\begin{align*}
n\Delta(G)^{2- 2^{2-c}} 2^{(c+3)(c-3)/2}&\leq n^{1+\frac{1}{2}(2- 2^{2-c})} 2^{(c+3)(c-2)/2}<n^{2- 2^{1-c}} 2^{(c+4)(c-1)/2},
\end{align*}
as desired.

For the $\Delta(G)> n^{1/2}$ case, the second term of the expression can be written as \\$(n(1-\frac{\Delta(G)}{n}))^{2- 2^{1-c}} 2^{(c+4)(c-1)/2}$. We use the following claim.
\begin{claim}
$(1-x)^k \leq 1-\frac{xk}{2}$ for any $0<x\leq 1/2$ and $1 \leq k \leq 2$
\end{claim}
\begin{proof} For any $y \in (0,1)$, $(1-y) \leq e^{-y}$ and $e^{-y} \leq 1-y/2$. Thus, $(1-x)^k \leq e^{-xk} \leq 1-xk/2$.

\end{proof}
Applying the claim with $x=\frac{\Delta(G)}{n}$ and $k=2-2^{1-c}$, it suffices to show that \[n^{2- 2^{1-c}} 2^{(c+4)(c-1)/2}\geq \Delta(G)^{2- 2^{2-c}} 2^{(c+3)(c-2)/2}\frac{n}{\Delta(G)} 2^c+n^{2- 2^{1-c}}\left(1-\frac{\Delta(G)(2- 2^{1-c})}{2n}\right) 2^{(c+4)(c-1)/2}.\]
or equivalently that
\[\Delta(G)^{2- 2^{2-c}} 2^{(c+3)(c-2)/2}\frac{n}{\Delta(G)} 2^c-n^{2- 2^{1-c}}\frac{\Delta(G)(2- 2^{1-c})}{2n} 2^{(c+4)(c-1)/2}\leq 0.\]
Simplifying the left-hand side of the above inequality and using the fact that $c \geq 1$:
\begin{align*}
&\Delta(G)^{2- 2^{2-c}} 2^{(c+3)(c-2)/2}\frac{n}{\Delta(G)} 2^c-n^{2- 2^{1-c}}\frac{\Delta(G)(2- 2^{1-c})}{2n}) 2^{(c+4)(c-1)/2}\\
&=n\Delta(G)^{1- 2^{2-c}} 2^{(c^2+c-6)/2}-n^{1- 2^{1-c}}\Delta(G)(1- 2^{-c}) 2^{(c^2+3c-4)/2}\\
&\leq 2^{(c^2+c-6)/2}(n\Delta(G)^{1- 2^{2-c}}-n^{1- 2^{1-c}}\Delta(G)(1-2^{-c})2^{c+1}) \\
& = 2^{(c^2+c-6)/2}(n\Delta(G)^{1- 2^{2-c}}-n^{1- 2^{1-c}}\Delta(G)(2^{c+1}-2)) \\
& \leq 2^{(c^2+c-6)/2}(n\Delta(G)^{1- 2^{2-c}}-n^{1- 2^{1-c}}\Delta(G)) \\
& = 2^{(c^2+c-6)/2}n^{1-2^{1-c}}\Delta(G)^{1- 2^{2-c}}(n^{ 2^{1-c}}-\Delta(G)^{2^{2-c}})  \leq 0.
\end{align*}
The last inequality holds because  $\Delta\geq n^{1/2}$.
\end{proof}

Like the proof of the initial bound (Theorem \ref{thm:init}), the proof of the improved bound (Theorem \ref{thm:imp}) also suggests an algorithm for generating the set of maximal cliques involving the preprocessing step of listing the set of all wedges in the graph. However, this algorithm is not asymptotically faster than Algorithm~\ref{alg} since its dependence on $n$ still includes $p(n,c)$ and we omit it.

\section{Lower bound}\label{sec:lb}
\begin{theorem}[restatement of Theorem~\ref{intro:lb}]\label{thm:lb}
For any positive integer $c$, we can construct graphs which are $c$-closed and with $\Omega(c^{-3/2}2^{c/2}n^{3/2})$ maximal cliques. 
\end{theorem}


\paragraph*{Construction.}
We suppose that $c$ is even and $n$ is a multiple of $c$. We can do this with only an absolute constant factor loss in the bound, which is allowable. We start with a graph $H$ on $v = 2n/c$ vertices with girth $5$ and the maximum possible number of edges, which is $\Omega(v^{3/2})$~\cite{girth5}. 

We construct our $c$-closed graph $G$ on $n$ vertices from $H$ in the following way.
For each vertex $x \in V(H)$, we replace it with a vertex set $U_x$ with $c/2$ vertices. Therefore, there are $|V(H)| \cdot c/2 = n$ vertices in $G$. The adjacency relation of $G$ is as follows. 
\begin{itemize}
\item Add all edges within each $U_x$ so that $U_x$ is a clique for all $x\in V(H)$.
\item For any edge $(x,y)$ of $H$, we place edges between the vertex sets $U_x$ and $U_y$ such that the bipartite graph between $U_x$ and $U_y$ consists of a complete bipartite graph minus a perfect matching. 
\item For any distinct and nonadjacent $x, y \in V(H)$, there are no edges between $U_x$ and $U_y$. 
\end{itemize}

Theorem~\ref{thm:lb} follows from the next two claims. 

\begin{claim}
The graph $G$ constructed is $c$-closed. 
\end{claim}
\begin{proof}

It suffices to check that for any two non-adjacent vertices in $G$, they have at most $c-1$ common neighbors.
By the construction, there are only two types of non-adjacent vertices:

\textbf{Case 1}: The non-adjacent pair $u,v \in V(G)$ are such that $u \in U_x, v\in U_y$ and $x, y \in V(H)$ are disitinct and non-adjacent in $H$.

In this case, there are no edges between $U_x, U_y$, and the common neighbors of $u, v$ are such that there is a vertex $z \in V(H)$ such that $(x,z), (y,z)$ are both edges in $H$. Since $H$ has girth 5, there is at most one such $z \in H$ given $x,y$, as otherwise $H$ would contain a $C_4$. Vertex $u \in V(G)$ is adjacent to exactly $|U_z| - 1 = (c/2)-1$ vertices in $U_z$, so $u,v$ can have at most $(c/2)-1$ common neighbors.

\textbf{Case 2}: The non-adjacent pair $u,v \in V(G)$ is such that $u \in U_x, v\in U_y$ and $x, y \in V(H)$ are adjacent in $H$.

In this case, $u$ and $v$ are adjacent to all other vertices in $U_x \cup U_y$, so they have $c-2$ common neighbors in $U_x \cup U_y$. Suppose for contradiction that $u,v$ have some other common neighbor $w$ and $w \in U_z$ for some $z \neq x, y$. This implies that $(w,x), (w,y)$ are both edges in $H$. However, $(x,y)$ is already an edge in $H$ by the assumption of this case. This implies that $H$ contains a triangle, which contradicts the fact that $H$ has girth 5. 

Combining both cases, we know that $G$ is $c$-closed. 
\end{proof}

\begin{claim}
There are $\Omega(c^{-3/2}2^{c/2}n^{3/2})$ maximal cliques in $G$.
\end{claim}
\begin{proof}
For any edge $(x,y)$ of $H$, picking one endpoint of each non-edge in $U_x \cup U_y$ gives a maximal clique.  Thus for each edge $(x,y)$ of $H$, there are exactly $2^{|U_x|} = 2^{c/2}$ maximal cliques. 

There are $\Omega(|V(H)|^{3/2}) = \Omega((2n/c)^{3/2})$ edges in $H$. As each of the maximal cliques obtained are distinct, we obtain $\Omega(2^{c/2} \cdot (2n/c)^{3/2})=\Omega(c^{-3/2}2^{c/2}n^{3/2})$ maximal cliques. 
\end{proof}

\section{Open problems and future directions}
\paragraph*{Direct improvement of our results}
\begin{itemize}
\item Determine the exact dependence on $n$ for the maximum possible number of of maximal cliques in a $c$-closed graph. We have proven (up to constant dependence on $c$) that this number is between $n^{3/2}$ and $n^{2-2^{1-c}}$.
\item Find a faster algorithm for listing the set of all wedges (induced 2-paths) in a $c$-closed graph (this would improve the runtime of Algorithm~\ref{alg}).
\end{itemize}

\paragraph*{Further exploration of c-closed graphs}
\begin{itemize}
\item Study the densest $k$-subgraph problem, a generalization of the clique problem, on $c$-closed graphs. The input to the problem is a graph $G$ and a parameter $k$, and the goal is to to find the subgraph of $G$ on $k$ vertices with the most edges. Unlike the clique problem, densest $k$-subgraph is NP-hard even for 2-closed graphs (more specifically, for graphs of girth 6)~\cite{raman2008short}. For general graphs, the best-known approximation algorithm has approximation ratio roughly $O(n^{1/4})$~\cite{bhaskara2010detecting} and under certain average-case hardness assumptions (concerning the planted clique problem), constant-factor approximation algorithms do not exist~\cite{alon2011inapproximability}.
\item Determine which other NP-hard problems are fixed-parameter tractable with respect to $c$.
\item Determine which problems in P have faster algorithms on $c$-closed graphs.
\end{itemize}

\paragraph*{Other model-free definitions of social networks}
\begin{itemize}
\item Explore other graph classes motivated by the well-established signatures of social networks (described in the introduction): heavy-tailed degree distributions, high triangle density, dense ``communities'', low diameter and the small world property, and triadic closure.
\item Determine other model-free definitions of social networks, for example, those motivated by 4-vertex subgraph frequencies. Ugander et al.~\cite{UgBaKl13} and subsequently Seshadhri ~\cite{JhSePi15} computed 4-vertex subgraph counts in a variety of social networks and the frequencies observed are far different than what one would expect from a random graph. In particular, social networks tend to have far fewer induced 4-cycles than random graphs.
\end{itemize}

\subparagraph*{Acknowledgements.}

We would like to thank Christina Gilbert for writing the code to calculate the $c$-closure and weak $c$-closure of networks in the SNAP data sets.

We would also like to thank Virginia Vassilevska Williams and Josh Alman for useful conversations about turning our bound into an algorithm. 

\bibliography{references}

\pagebreak
\appendix
\section{Appendix: Algorithm to generate all maximal cliques}\label{app:alg}
In this section we prove Theorem~\ref{thm:alg}, restated below. Recall that $p(n,c)$ denotes the time to list all wedges (induced 2-paths) in a $c$-closed graph on $n$ vertices.

\begin{theorem}[restatement of Theorem~\ref{thm:alg}]\label{app:algthm}
A superset of the maximal cliques in any $c$-closed graph can be generated in time $O(p(n,c)+3^{c/3}n^2)$. The exact set of all maximal cliques in any $c$-closed graph can be generated in time $O(p(n,c)+3^{c/3}2^ccn^2)$.
\end{theorem}

Before proving Theorem~\ref{app:algthm}, we give a bound on $p(n,c)$, which follows from a result of G\k{a}sieniec, Kowaluk, and Lingas~\cite{gkasieniec2009faster} about computing \emph{witnesses} of boolean matrix multiplication. If $C$ is the boolean matrix product of $A$ and $B$, a \emph{witness} of entry $C[i,j]$ is an index $l$ such that $A[i,l]=B[l, j]=1$.

\begin{lemma}[Theorem 1 in ~\cite{gkasieniec2009faster}]
If $C$ is the boolean matrix product of two $n\times n$ matrices, we can report all witnesses of all entries of $C$ that have at most k witnesses in expected time $O(n^{2+o(1)}k + k^{(3-\omega-\alpha)/(1-\alpha)}n^\omega  + n^\omega \log n)$ where $\omega$ is the matrix multiplication exponent and $\alpha$ is the supremum of the set of $r\in [0,1]$ such that multiplying an $n \times n^r$ matrix by an $n^r \times n$ takes time $O(n^{2+o(1)})$.
\end{lemma}

Let $A$ be the adjacency matrix of a $c$-closed graph on $n$ vertices and let $C$ be the boolean matrix product $A^2$. Then $C[i,j]=1$ if and only if vertices $i$ and $j$ have at least one common neighbor. Since $G$ is $c$-closed, all $C[i,j]$ for non-adjacent $i,j$ have at most $c-1$ witnesses. Thus, we have the following corollary.

\begin{corollary}
All wedges in a $c$-closed graph can be listed in time $p(n,c) = O(n^{2+o(1)}c + n^\omega c^{(3-\omega-\alpha)/(1-\alpha)} + n^\omega \log n)$.
\end{corollary}

\begin{proof}[Proof of Theorem~\ref{app:algthm}]

The algorithm follows naturally from the proof of Theorem~\ref{thm:init} with two additional ingredients:
\begin{itemize}
\item A preprocessing step to enumerate all wedges in the graph speeds up the later process of finding the intersection of the neighborhoods of two vertices (i.e. $N(u) \cap N(v)$ from the proof of Theorem~\ref{thm:init}). 
\item An algorithm of Tomita, Tanaka, and Takahashi~\cite{tomita} that generates all maximal cliques in any graph in time $O(3^{n/3})$. We apply this to the recursive calls on the small graphs $G[N(u) \cap N(v)]$ from the proof of Theorem~\ref{thm:init}). Let \textproc{Cliques}($G$) be a call to this algorithm.
\end{itemize}

The output of \textproc{Cliques}($G$) does not explicitly list every maximal clique, as this could take $\Omega(3^{n/3}n)$ time e.g. for a complete $\frac{n}{3}$-partite graph). Instead, the output of \textproc{Cliques}(G) is a forest $F$ where each node represents a vertex in $G$ and the collection of nodes on any path from root to leaf in $F$ form a maximal clique in $G$. The output of our algorithm will be of the same form. For any leaf $l\in V(F)$ let $K(l)$ be the maximal clique in $G$ on the set of vertices along the path from $l$ to the root of its tree in $F$.

See Algorithm~\ref{alg} for pseudocode describing our algorithm that generates a superset of the maximal cliques in a $c$-closed graph. The correctness of Algorithm~\ref{alg} follows from the proof of Theorem~\ref{thm:init}.
\begin{algorithm}
\caption{}\label{alg}
\begin{algorithmic}[1]
\Procedure {Preprocess(G)}{}
\State enumerate all wedges in $G$
\State $M\gets$ mapping such that $M[v]$ is the set of wedges with $v$ as an endpoint
\EndProcedure
\Procedure{CClosedCliques(G)}{}
\If {$|V(G)|=1$}
\State \Return $G$
\EndIf
\State fix an arbitrary vertex $v \in V(G)$
\State $F\gets \textproc{CClosedCliques}(G\setminus \{v\})$\label{recurse}
\For {each leaf $l$ in $F$ with $K(l)\subseteq N(v)$}\label{dfs}
\State add $v$ to $F$ as a child of $l$\label{dfs2}
\EndFor
\State $E'\gets$ the set of all edges in some wedge in $M[v]$ except for those edges incident to $v$
\State $H \gets G(V, E')$
\For {each vertex $u\in G\setminus(N(v)\cup \{v\})$}
\State $F_u\gets \textproc{Cliques}(G[N_H(u)])$
\State add $v$ to $F_u$ with an edge between $v$ and every current root in $F_u$
\EndFor
\State \Return $F\cup(\cup_u F_u)$
\EndProcedure
\end{algorithmic}
\end{algorithm} 

\paragraph{Runtime analysis.} \textproc{Preprocess} takes time $p(n,c)$. Let $T'(n,c)$ be the runtime of \textproc{CClosedCliques}. Line~\ref{recurse} makes a recursive call to \textproc{CClosedCliques}($G\setminus \{v\}$) and since $G\setminus \{v\}$ is $c$-closed, this takes time $T'(n-1,c)$. Lines~\ref{dfs} and~\ref{dfs2} can be implemented as a depth-first search of $F$ that traverses only the subtrees rooted at a vertex in $N(v)$. Next, $E'$ is the set of all edges between $N(v)$ and $G\setminus(N(v)\cup \{v\}$ and given $M$, it takes constant time per edge in $E'$ to find all of the edges in $E'$, construct $H$, and find $N_H(u)$ for all $u$. The number of edges between $E'$ is at most $n(c-1)$ since $N_H(u)$ is at most $c-1$. Then, each call to \textproc{Cliques} runs in time $O(3^{c/3})$. Thus, $T'(n,c)=T'(n-1,c)+O(3^{c/3}n)$, so $T'(n,c)= O(3^{c/3}n^2)$ and Algorithm~\ref{alg} runs in time $O(p(n,c)+3^{c/3} n^2)$.

The reason that \textproc{CClosedCliques} lists a superset of the maximal cliques rather than the exact set is the following. There could be a clique $K$ in $N(v)\cap N(u) \cap N(w)$ for some $u$ and $w$ such that $K$ is maximal in $N(v)\cap N(u)$ but not maximal in $N(v)\cap N(w)$. In this case $K\cup\{v\}$ will be reported in the output even though it is not a maximal clique. To list the exact set of maximal cliques, we make the following addition to the procedure \textproc{CClosedCliques} right before returning. Add every clique in $\cup_u F_u$ to a hash set $S$. Then iterate through every clique $K$ in $\cup_u F_u$ in order from largest to smallest, and check whether any subset of of $K$ is in $S$. Note that the number of vertices in $K$ is at most $c$. This increases the runtime to $O(p(n,c)+3^{c/3}2^ccn^2)$.

\end{proof}

\section{Appendix: Weakly $c$-closed graphs}\label{app:weak}


In this section we give experimental results regarding the $c$-closure and weak $c$-closure of well-studied social networks, an algorithm to compute the smallest value $c$ such than a given graph is weakly $c$-closed, and an equivalent definition of weakly $c$-closed graphs.

Recall the definition of a weakly $c$-closed graph.
\begin{definition}
Given $c$, a \emph{bad pair} is a non-adjacent pair of vertices with at least $c$ common neighbors.
\end{definition}

\begin{definition}\label{app:def1}
A graph is \emph{weakly $c$-closed} if there exists an ordering of the vertices $\{v_1,v_2,\dots,v_n\}$ such that for all $i$, $v_i$ is in no bad pairs in the graph induced by $\{v_i,v_{i+1},\dots,v_n\}$.
\end{definition}

\subsection{Experimental results}
Table~\ref{tab} shows the $c$-closure and weak $c$-closure of some well-studied social networks.
\begin {table}[ht]
\begin{center}
\begin{tabular}{| l | l | l | l | l | }
  		\hline
  &$n$&	$m$&	$c$&	weak $c$ \\
  \hline
  email-Enron&	36692&	183831&	161&	34 \\
  \hline
  p2p-Gnutella04	&10876	&39994	&24&	8 \\
  \hline  
wiki-Vote&7115	&103689&	420&	42\\
\hline 
ca-GrQc&5242	&14496&	41&	9\\
\hline 

\end{tabular}
\end{center}
\caption{The $c$-closure and weak $c$-closure of well-studied social networks. The data sets are from the Stanford Network Analysis Platform (SNAP)~\cite{snapnets} and are each categorized as either a social network, communication network, collaboration network, or internet peer-to-peer network. For the networks with directed edges, we analyze the underlying undirected graphs. For each network $G$, $n$ is the number of vertices, $m$ is the number of edges, $c$ is the smallest value $c$ such that $G$ is $c$-closed, and ``weak $c$'' is the smallest value $c$ such that $G$ is weakly $c$-closed.}\label{tab}
\end{table}

\subsection{Computing weak $c$-closure} 
To get an algorithm for computing the weak $c$-closure of a graph, we first show that the ordering of the vertices in the definition of weakly $c$-closed can be chosen greedily. In the following definition, we define a \emph{valid} ordering of the vertices as one that satisfies the definition of a weakly $c$-closed graph.
\begin{definition}\label{app:valid}
We say that an ordering of the vertices $\{v_1,v_2,\dots,v_n\}$ is \emph{valid} if for all $i$, the graph induced by $\{v_i,v_{i+1},\dots,v_n\}$ contains a vertex in no bad pairs.
\end{definition}
\begin{claim}\label{claim:greed}
In any weakly $c$-closed graph, a valid vertex ordering can be chosen as follows: for all $i$ in order from 1 to $n$, let $v_i$ be any arbitrary vertex in no bad pairs with respect to the graph induced by $\{v_i,v_{i+1},\dots,v_n\}$.
\end{claim}
The following lemma will be useful.
\begin{lemma}\label{lem:subgraph}
If a vertex $v$ is in no bad pairs with respect to a graph $G$, then $v$ is also in no bad pairs with respect to any induced subgraph $H\subseteq G$.
\end{lemma}
\begin{proof}
Fix a vertex $u\in V(G)$ not adjacent to $v$. By definition, $|N_G(v)\cap N_G(u)|<c$. The set of common neighbors of $u$ and $v$ in $H$ is a subset of the set of common neighbors of $u$ and $v$ in $G$ so $|N_H(v)\cap N_H(u)|<c$. Thus, $v$ is in no bad pairs with respect to $H$.
\end{proof}

\begin{proof}[Proof of Claim~\ref{claim:greed}]
Fix a valid ordering $O=\{v_1,v_2\dots,v_n\}$ of the vertices. It suffices to show that if $v_j$ is a vertex in no bad pairs, then the ordering $O'=\{v_j,v_1,v_2,\dots,v_{j-1},v_{j+1},\dots,v_n\}$ is also a valid ordering. By definition, for all $i$, $v_i$ is in no bad pairs with respect to the graph induced by $\{v_i,v_{i+1},\dots,v_n\}$ so by Lemma~\ref{lem:subgraph} for $i<j$, $v_i$ is also in no bad pairs with respect to the graph induced by $\{v_i,v_{i+1},\dots,v_{j-1},v_{j+1},\dots,v_n\}$.
\end{proof}
The fact that the vertex ordering can be chosen greedily suggests an $O(n^3)$ time algorithm for computing the weak $c$-closure of any graph $G$. First, square the adjacency matrix of $G$ to find the number of neighbors shared by each pair of vertices. Then, repeatedly find the minimum value of $c$ such that there exists a vertex $v$ in no bad pairs and remove $v$, updating the matrix. The $c$ value may be different at each iteration. When the graph is empty, return the maximum such value $c$ over all iterations. 

\subsection{Equivalent definition of weakly $c$-closed}

\begin{definition}\label{app:def2}
A graph $G$ is \emph{weakly $c$-closed} if every induced subgraph of $G$ contains a vertex in no bad pairs.
\end{definition}

\begin{claim}\label{claim:defs}
Definitions \ref{app:def2} and \ref{app:def1} are equivalent.
\end{claim}

\begin{proof}
Recall the definition of a \emph{valid} vertex ordering (Definition~\ref{app:valid}). 

If $G$ satisfies Definition \ref{app:def2} then one can construct a valid vertex ordering by simply iteratively selecting and deleting a vertex in no bad pairs with respect to the current graph.

Now suppose $G$ satisfies Definition \ref{app:def1} with valid elimination ordering $O$. Given any subgraph $H\subseteq G$ consider the vertex $v_i$ in $H$ that appears in the elimination ordering $O$ before all other vertices in $H$. By definition, $v_i$ is in no bad pairs with respect to the graph induced by $\{v_i,v_{i+1},\dots,v_n\}$. $H$ is a subgraph of this graph so by Lemma~\ref{lem:subgraph}, $v_i$ is also in no bad pairs with respect to $H$.
\end{proof}



\begin{remark}
We argue that this definition of weakly $c$-closed is tight in the following sense. Note that if a graph $G$ is such that every subgraph $H\subseteq G$ has at most $\lfloor\frac{|V(H)|-1}{2}\rfloor$ bad pairs, then all $H\subseteq G$ must contain a vertex in no bad pairs so $G$ must be weakly $c$-closed. One might hope that the clique listing problem remains fixed-parameter tractable for some new definition of weakly $c$-closed of the form ``every subgraph of $G$ has at most $k$ bad pairs'' for some $k\geq \frac{|V(H)|}{2}$. To see why this is impossible, let $G$ be a complete $n/2$-partite graph i.e. the complement of a perfect matching. Then, the only bad pairs in $G$ are the endpoints of the $n/2$ non-edges so every subgraph $H\subseteq G$ has at most $\frac{V(H)}{2}$ bad pairs, yet $G$ has $2^{n/2}$ maximal cliques (choose one endpoint of each non-edge).

From the other side, such a definition is trivial for $k<\frac{|V(H)|}{2+c}$. That is, if a graph $G$ is not $c$-closed, then there exists a subgraph $H\subseteq G$ with at least $\frac{|V(H)|}{2+c}$ bad pairs. Specifically, consider the graph induced by a non-adjacent pair of vertices and $c$ of their common neighbors.

\end{remark}

\section{Appendix: Generalizations of $c$-closed graphs and weakly $c$-closed graphs}\label{app:general}
\subsection{$A$-bounded graphs}

In this subsection, we define \emph{$A$-bounded} graphs, a generalization of weakly $c$-closed graphs. In Definition~\ref{app:def2} of weakly $c$-closed graphs we require that every subgraph has at least one vertex $v$ in no bad pairs, while here we do allow $v$ to be in bad pairs but restrict the sizes of the common neighborhoods of $v$ with its non-neighbors.

\begin{definition}\label{def:Abound}
A graph $G$ is $A$-bounded if for any subgraph $H\subseteq G$, there is a vertex $v\in V(H)$ such that 
\[  \sum_{w \in V(H), w \notin N_H(v)} 3^{|N_H(v) \cap N_H(w)|/3} \leq A.\]
\end{definition}

Note that weakly $c$-closed graphs are $(n-2)3^{c/3}$-bounded.

Like weakly $c$-closed graphs, $A$-bounded graphs also have two equivalent interpretations: one in terms of subgraphs of $G$ (Definition~\ref{def:Abound}) and another in terms of an ordering of the vertices.

\begin{lemma}\label{lem:Abound}
A graph $G$ is $A$-bounded if and only if there is an ordering of the vertices of $G$: $v_1,v_2, \dots, v_n$ such that $\{v_i,v_{i+1}, \dots, v_i\}$ is $A$-bounded for each $i$. 
\end{lemma}
The proof of Lemma \ref{lem:Abound} is analogous to that of Lemma~\ref{claim:defs} and we omit it.


\begin{theorem}\label{thm:Abound}
The number of maximal cliques in an $A$-bounded graph with $n$ vertices is at most $nA$. 
\end{theorem}

Theorem \ref{thm:Abound} follows as a corollary from the following theorem which generalizes Theorem~\ref{thm:init}.

\begin{theorem}\label{thm:gen}
Let $v_1, v_2, \dots, v_n$ be an arbitrary ordering of the vertices of $G$. Let $S=\{\{v_i,v_j\}\not\in E(G)\big\vert| N(v_i) \cap N(v_j) \setminus \{v_1, \dots, v_i\}| \neq 0\}$ i.e. the set of pairs of non-adjacent vertices whose common neighborhood contains at least one vertex that comes later than $i$ in the ordering. Then the number of maximal cliques in $G$ is at most 
 \[  \sum_{i=1}^{n-1}  \sum_{\substack{j=i+1,\\\{v_i,v_j\}\in S}}^n  3^{ | N(v_i) \cap N(v_j) \setminus \{v_1, \dots, v_i\}|  /3} .
 \]
\end{theorem}




\begin{proof}
We generalize the proof of Theorem \ref{thm:init}. Let $G_i$ be the induced subgraph of $G$ on the vertices $\{ v_{i+1},\dots, v_n\}$.
Let $F(G_i)$ be the number of maximal cliques in $G_i$. We will write a recurrence for $F(G_i)$ in terms of subgraphs of $G$. 

In $G_i$, every maximal clique $K$ is one of the following types: 
\begin{enumerate}
\itemsep0em 
\item $v_{i+1}\not\in K$ and $K$ is maximal in $G_{i+1}$.
\item $v_{i+1}\in K$ and $K\setminus\{v_{i+1}\}$ is maximal in $G_{i+1}$.
\item $v_{i+1}\in K$ and $K$ is not maximal in $G_{i+1}$. 
\end{enumerate} Every type 1 or 2 maximal clique can be obtained by starting with a clique maximal in $G_{i+1}$ and adding $v_{i+1}$ if possible so the number of such cliques is at most $F(G_{i+1})$. 

Now we bound the number of maximal cliques of type 3. Let $K$ be such a clique. Since $K$ is not maximal in $G_{i+1}$, there must exist a vertex $u\in G_{i+1}\setminus N(v_{i+1})$ such that $K\setminus v_{i+1}\subseteq N(u)$. For each $u\in G_{i+1}\setminus N(v_{i+1})$, the total number of maximal cliques $K$ of type 3 such that $K\setminus v_{i+1}\subseteq N(u)$ is at most $3^{|N(v_{i+1})\cap N(u) \setminus \{ v_1, \dots, v_{i}\}|/3}$. 
Therefore, we have 
\[ f(G_i) \leq f(G_{i+1}) + \sum_{u \in V(G_{i+1})\setminus N(v_{i+1}) }  3^{|N(v_{i+1})\cap N(u) \setminus \{ v_1, \dots, v_{i}\}|/3}.
\]
By iterating the above inequality until $G_0=G$, we obtain the desired inequality. 
\end{proof}

\subsection{Graphs with given common neighborhood statistics}

In the definition of a $c$-closed graph, we require that no pair of non-adjacent vertices has $c$ common neighbors, while here we allow such pairs and consider the number of pairs with exactly $i$ common neighbors for all $i$.

\begin{definition}
For any integer $0 \leq i \leq n-2$, let $p(i)$ be the number of pairs of non-adjacent vertices with exactly $i$ common neighbors. 
\end{definition}
Note that $c$-closed graphs have $p(i) = 0$ for all $i > c$ and $\sum_{i>c } p(i)$ counts the total number of bad pairs in $G$. 


\begin{theorem} \label{thm:stat} 
The number of maximal cliques in any graph is at most $\sum_{i > 0} 8p(i) \frac{3^{i/3}}{i+2} .$
\end{theorem}
\begin{proof}
Directly from Theorem~\ref{thm:gen}, we have that the number of maximal cliques in any graph $G$ is at most
 \begin{align*} & \sum_{i=1}^{n-1}  \left(  \sum_{j=i+1}^n  1_{v_i, v_j \text{not adjacent}} \cdot 3^{ | N(v_i) \cap N(v_j) \setminus \{v_1, \dots, v_i\}|  /3}\right)  \\
 \leq& \sum_{i=1}^{n-1}  \left(  \sum_{j=i+1}^n  1_{v_i, v_j \text{not adjacent}} \cdot 3^{ | N(v_i) \cap N(v_j)|  /3}\right)  \\   
 = & \sum_{i \geq 0} p(i) 3^{i/3}.
 \end{align*}
 
 We can do better by applying Theorem~\ref{thm:gen} using a uniformly random ordering of the vertices in $G$. The number of maximal cliques in $G$ is at most
 \begin{align*}
\mathbb{E}\left[ \sum_{i=1}^{n-1}  \left(  \sum_{j=i+1}^n  1_{v_i, v_j \text{not adjacent}} \cdot 3^{ | N(v_i) \cap N(v_j) \setminus \{v_1, \dots, v_i\}|  /3}\right)\right],
 \end{align*}
 where the expectation is over all ordering of the vertices in $G$ where each order appears uniformly at random. 
 
 By linearity of expectation, for any two vertices $u, v$ which are not adjacent, we want to compute \begin{equation}\mathbb{E}\left[3^{|N(v)\cap N(u) \setminus \{ \text{vertices come before } u, v \text{ in the ordering}  \}|/3} \cdot 1_{u,v \text{ not adjacent}}\right]. \label{eq:exp} \end{equation}
 Let $s=|N(u)\cap N(v)|$. This expectation can simply be written as 
 \[ \sum_{k=0}^s 3^{(s-k)/3} \text{Pr}\left( \text{exactly } k \text{ vertices come before }u, v \text{ in the ordering} \right).
 \]
 The probability that there are exactly $k$ vertices before $u, v$ in the random ordering is the probability that when permute the $|N(u)\cap N(v)| = s$ vertices together with $u,v$, there are exactly $k$ vertices from $N(u) \cap N(v)$ that come before $u$ and $v$. In other words, it means that in the random permutation of size $s +2$, one of $u,v$ comes at the $k+1$-th position, and the other one comes after the $k+1$-th position. This probability is $\frac{2(s+2-(k+1))}{(s+1)(s+2)}$. The denominator is the number of ways to place $u, v$ in the $s+2$ positions; while the numerator is the probability that one of $u,v$ is at the $(k+1)$-th position, and the other one is at the last $(s+2) - (k+1)$ positions. 
 Therefore
  \[ \sum_{k=0}^s 3^{(s-k)/3} \text{Pr}\left( \text{exactly } k \text{ vertices come before }u, v \text{ in the ordering} \right) = \sum_{0 \leq k \leq s}3^{(s-k)/3} \frac{2(s+1-k)}{(s+1)(s+2)}.
 \]
 Therefore, (\ref{eq:exp}) is equal to
 \begin{align*}
 & \sum_{u, v\text{ not adjacent}} \sum_{s=1}^n \sum_{0 \leq k \leq s}3^{(s-k)/3} \frac{2(s+1-k)}{(s+1) (s+2)}\cdot \mathbf{1}_{|N(u)\cap N(v)|=s} \\
 = & \sum_{s=0}^n p(s) \left(\sum_{0 \leq k \leq s}3^{(s-k)/3} \frac{2(s+1-k)}{(s+1) (s+2)}\right).
\end{align*}
The last equality holds by rearranging the terms and combining all the non-adjacent pairs $u,v$ with $|N(u)\cap N(v)| = s$. For each fixed $s$, by definition there are $p(s)$ such pairs $u,v$. 
 Then by some routine computation, we obtain the desired bound. 
\end{proof}

Note that Theorem~\ref{thm:stat} implies Theorem~\ref{thm:init}. This is because if $G$ is $c$-closed, $p(i) = 0$ for all $i > c$ so by Theorem~\ref{thm:stat} the number of maximal cliques in $G$ is at most  $\sum_{i \leq c} p(i) 3^{i/3} \leq \binom{n}{2} 3^{c/3}$ since $\sum p(i)$ is at most the total number of pairs of vertices in $G$. 

 
\begin{lemma}
Theorem~\ref{thm:stat} is tight up to a constant.
\end{lemma}
\begin{proof}
Let $G$ be the complete $n/3$-partite graph i.e. the complement of a disjoint union of triangles. $G$ has $3^{n/3}$ maximal cliques. Also, $p(n-3) = n$ while $p(i) = 0$ for all other $i$ so Theorem~\ref{thm:stat} gives $8n\cdot\frac{ 3^{(n-3)/3}}{n-3+2} \simeq \frac{8}{3} \cdot 3^{n/3}$. 
\end{proof}

\section{Appendix: Cliques in $c$-closed, $K_k$-free graphs}\label{app:kk}

We leave as an open question the exact exponent of $n$ (between $3/2$ and $2-2^{1-c}$) in the expression for the maximum number of maximal cliques in a $c$-closed graph. In this section we show that if there exist $c$-closed, $K_k$-free graphs that achieve our upper bound (i.e. have $\Omega(n^{2-2^{1-c}}$ maximal cliques), they must have a certain distribution of clique sizes. For all positive integers $j$, let $n_j(G)$ denote the number of (not necessarily maximal) $K_j$'s in $G$.

\begin{theorem}\label{thm:kkfree}
For all constant integers $c\geq 1$, $k\geq 2$, and $j\geq 2$, if $G$ is a $c$-closed, $K_k$-free graph on $n$ vertices,
\begin{enumerate}
\item $n_2(G)= O(n^{3/2})$ 
\item for all $j>2$, $n_{j+1}(G)= O(n_j(G)^{1/2}n)$. 
\end{enumerate}
\end{theorem}

Remarks: Solving the recurrence in the theorem statement reveals that for all constant integers $j\geq 2$, $G$ has $O(n^{2-2^{1-j}})$ $j$-cliques. Notice the similarity of this result to Theorem~\ref{thm:imp}. For example, Theorem~\ref{thm:imp} says that 2, 3, and 4-closed graphs have $O(n^{3/2})$, $O(n^{7/4})$, and $O(n^{15/8})$ maximal cliques respectively, while Theorem~\ref{thm:kkfree} says that for all constants $c$ and $k$, any $c$-closed, $K_k$-free graph has $O(n^{3/2})$ edges, $O(n^{7/4})$ triangles, $O(n^{15/8})$ $K_4$'s etc. Furthermore, it says that $G$ can only have $\Theta(n^{2-2^{1-j}})$ $j$-cliques if it has $\Theta(n^{2-2^{1-i}})$ $i$-cliques for all $i<j$.

\begin{proof}[Proof of Theorem~\ref{thm:kkfree}] $ $
\begin{enumerate}
\item We proceed by induction on $k$.

{\bf Base case:} If $k=2$, $G$ has no edges and the result is trivial. 

{\bf Inductive hypothesis:} Suppose $n_2(G)= O(n^{3/2})$ for $k=i$.

{\bf Inductive step:} Let $k=i+1$. Let $d$ be the average degree of a vertex in $G$. 

For all $v\in V(G)$, $G[N(v)]$ is $c-1$-closed and $K_{i}$-free so by the inductive hypothesis, $G[N(v)]$ has $O(d(v)^{3/2})$ edges and thus $\Omega(d(v)^2)$ non-edges. Since each non-edge is in at most $c-1$ neighborhoods and $c$ is a constant, the total number of non-edges in $G$ is 
\begin{align*}
\Omega(\sum_{v\in V(G)}d(v)^2/(c-1))&=\Omega(\sum_{v\in V(G)}d^2)\\&=\Omega(nd^2).
\end{align*}
$G$ can have no more than $n^2$ non-edges so $nd^2\in O(n^2)$, or equivalently $d= O(n^{1/2})$, so $n_2(G)= O(n^{3/2})$.

\item We generalize the proof of part 1 of the theorem. Fix $j$. Let $\cJ$ be the set of all $K_j$'s in $G$. For a given $J\in \cJ$, let $N(J)=\bigcap_{v\in V(J)} N(v)$. Let $s$ be the average value of $|N(J)|$ over all $J\subseteq \cJ$. For all $J\in \cJ$, $G[N(J)]$ is $(c-j)$-closed and $K_{k-j}$-free so by part 1 of the theorem, $G[N(J)]$ has $O(|N(J)|^{3/2})$ edges and thus $\Omega(|N(J)|^2)$ non-edges. Since each pair of non-adjacent vertices has at most $c-1$ common neighbors, each such pair is in $N(J)$ for at most $\binom{c-1}{j}$ choices of $J\in \cJ$. Then, since $c$ is constant, the total number of non-edges in $G$ is 
\begin{align*}
\Omega\left(\sum_{J\in \cJ}\frac{|N(J)|^2}{\binom{c-1}{j}}\right)&=\Omega\left(\sum_{J\in \cJ}|N(J)|^2\right)\\&=\Omega(\sum_{J\in \cJ}s^2)\\&=\Omega(n_j(G)s^2).
\end{align*}
$G$ can have no more than $n^2$ non-edges so $n_j(G)s^2= O(n^2)$, or equivalently,\begin{align}\label{non-edges}s = O\left(\frac{n}{n_j(G)^{1/2}}\right). \end{align}
An upper bound on $n_{j+1}(G)$ is the sum over all $J\in \cJ$ of the number of $K_{j+1}$'s that $J$ is in; that is, $n_{j+1}(G)\leq \sum_{J\in \cJ} |N(J)|=n_j(G)s= O(n_j(G)^{1/2}n)$ by Equation~\ref{non-edges}.
\end{enumerate}
\end{proof}

\end{document}